\documentclass[12pt,a4paper]{article}
\usepackage{amsmath}
\usepackage{amssymb}
\usepackage{t1enc}
\usepackage[latin1]{inputenc}
\usepackage[english]{babel}
\usepackage{amsfonts}
\usepackage{latexsym}
\usepackage{bm}
\usepackage{setspace}
\usepackage[toc,page]{appendix}
\usepackage{graphicx}
\usepackage{enumerate}
\usepackage[numbers,sort&compress]{natbib}
\usepackage{tikz}
\usepackage{soul}
\usepackage[normalem]{ulem}

\usepackage{theoremref}
\usepackage{lineno}

\usepackage{mathtools}

\newtheorem{theorem}{Theorem}%[section]
\newtheorem{prop}{Proposition}
\newtheorem{lemma}{Lemma}
\newtheorem{cor}{Corollary}

\newtheorem{rem}{Remark}
\newtheorem{defn}{Definition}
\newtheorem{construction}{Construction}

\allowdisplaybreaks

\begin{document}
\title{Results on three problems on isolation of graphs \medskip\medskip
}

\author{Peter Borg\\[5mm]
Department of Mathematics \\
Faculty of Science \\
University of Malta\\
Malta\\
\texttt{peter.borg@um.edu.mt} 
\and 
Yair Caro \\ [5mm]
Department of Mathematics \\
University of Haifa--Oranim \\
Israel \\
\texttt{yacaro@kvgeva.org.il}
}

\date{}%{\today} 
\maketitle

\begin{abstract}
The graph isolation problem was introduced by Caro and Hansberg in 2015. It is a vast generalization of the classical graph domination problem and its study is expanding rapidly. In this paper, we address a number of questions that arise naturally. Let $F$ be a graph. We show that the $F$-isolating set problem is NP-complete if $F$ is connected. We investigate how the $F$-isolation number $\iota(G,F)$ of a graph $G$ is affected by the minimum degree $d$ of $G$, establishing a bounded range, in terms of $d$ and the orders of $F$ and $G$, for the largest possible value of $\iota(G,F)$ with $d$ sufficiently large. We also investigate how close $\iota(G,tF)$ is to $\iota(G,F)$, using domination and, in suitable cases, the Erd\H os--P\' osa property.
\end{abstract}

\section{Introduction} \label{Introsection}
The graph isolation problem was introduced by Caro and Hansberg in \cite{CaHa15, CaHa17}. It is a natural and appealing generalization of the classical graph domination problem \cite{C, CH, HHS, HHS2, HL, HL2}. As was the case of domination, isolation has become a very active field of investigation and is rapidly expanding in various directions. The seminal paper \cite{CaHa17} addressed this problem from various angles. It established many results and posed a number of pertinent problems, a few of which are solved in \cite{Borgcycles, BFK, Borgrsc}. It has motivated an accelerating number of results, most of which are referenced in \cite{Borgrsc, Borgisdom}. The scope of this paper is to address mainly three broad questions that are central to the study at this point in its development. The first concerns the complexity of the problem. The second is how the minimum degree of a graph affects the isolation parameter considered. The third is how an isolation parameter of a disconnected graph relates to other isolation parameters of the graph. To the best of the authors' knowledge, isolation of disconnected graphs has not yet been investigated. 

For standard terminology in graph theory, we refer the reader to \cite{West}. Most of the notation and terminology used here is defined in \cite{Borgcycles}. The set of positive integers is denoted by $\mathbb{N}$. For $n \in \{0\} \cup \mathbb{N}$, $[n]$ denotes the set $\{i \in \mathbb{N} \colon i \leq n\}$. Note that $[0]$ is the empty set $\emptyset$. Arbitrary sets and graphs are taken to be finite.  For a set $X$, ${X \choose k}$ denotes the set of $k$-element subsets of $X$. Every graph $G$ is taken to be \emph{simple}, that is, its vertex set $V(G)$ and edge set $E(G)$ satisfy $E(G) \subseteq {V(G) \choose 2}$. Unless stated otherwise, it is to be assumed that an arbitrary graph is not the \emph{null graph}, that is, its vertex set is non-empty. We may represent an edge $\{v,w\}$ by $vw$. 

If $D \subseteq V(G) = N[D]$, then $D$ is called a \emph{dominating set of $G$}. The size of a smallest dominating set of $G$ is called the \emph{domination number of $G$} and is denoted by $\gamma(G)$. If $\mathcal{F}$ is a set of graphs and $F$ is a copy of a graph in $\mathcal{F}$, then we call $F$ an \emph{$\mathcal{F}$-graph}. A subset $D$ of $V(G)$ is called an \emph{$\mathcal{F}$-isolating set of $G$} if $N[D]$ intersects the vertex sets of the $\mathcal{F}$-graphs contained by $G$. Thus, $D$ is an $\mathcal{F}$-isolating set of $G$ if and only if $G - N[D]$ contains no $\mathcal{F}$-graph. %It is to be assumed that $(\emptyset, \emptyset) \notin \mathcal{F}$.
The size of a smallest $\mathcal{F}$-isolating set of $G$ is called the \emph{$\mathcal{F}$-isolation number of $G$} and is denoted by $\iota(G, \mathcal{F})$. If $\mathcal{F} = \{F\}$, then we may replace $\mathcal{F}$ in these defined terms and notation by $F$. Clearly, $D$ is a dominating set of $G$ if and only if $D$ is a $K_1$-isolating set of $G$. Thus, $\gamma(G) = \iota(G, K_1)$.

One of the earliest domination results is the upper bound $n/2$ of Ore \cite{Ore} on the domination number of any connected $n$-vertex graph $G$ with $n \geq 2$ (see \cite{HHS}). While deleting the closed neighbourhood of a dominating set yields the graph with no vertices, deleting the closed neighbourhood of a $K_2$-isolating set yields a graph with no edges. A $K_2$-isolating set is also called a \emph{vertex-edge dominating set} \cite{BCHH,Lewis,LHHF,Peters,TJ,Z}. Caro and Hansberg~\cite{CaHa17} proved that if $G$ is a connected $n$-vertex graph with $n \geq 3$, then $\iota(G, K_2) \leq n/3$ unless $G$ is a $5$-cycle. This solved a problem in \cite{BCHH} and was proved independently by \.{Z}yli\'{n}ski \cite{Z}. One of the main problems posed by Caro and Hansberg in \cite{CaHa17} essentially was to determine a best possible upper bound on $\iota(G,K_k)$ for any $k \geq 1$. This was solved by Borg, Fenech and Kaemawichanurat~\cite{BFK}, who proved that $\iota(G, K_k) \leq n/(k+1)$ unless $G$ is a $k$-clique or $k = 2$ and $G$ is a $5$-cycle. The bound of Ore is the case $k = 1$, and the bound of Caro and Hansberg and of \.{Z}yli\'{n}ski is the case $k = 2$. The graphs attaining the bound $n/(k+1)$ on $\iota(G, K_k)$ are determined in \cite{FJKR, PX} for $k = 1$, in \cite{BG, LMS} for $k = 2$, in \cite{CCZ} for $k = 3$, and in \cite{CCZ2} for $k \geq 4$. Let $\mathcal{C}$ be the set of cycles. Solving another problem in \cite{CaHa17}, Borg \cite{Borgcycles} proved that $n/4$ is a best possible upper bound on $\iota(G,\mathcal{C})$ unless $G$ is a $3$-cycle (\cite{Borgcon} provides a stronger version). A common generalization to the bounds above was recently provided by Borg in \cite{Borgrsc}. Each of \cite{Borgrc2,BFK2,CZ,ZW2} provides a sharp upper bound on $\iota(G, K_k)$ or $\iota(G, \mathcal{C})$ in terms of the size $|E(G)|$ of $G$.

As indicated above, this paper addresses important facets of the isolation problem for general frameworks. In the next section, we show that for any connected graph $F$, the problem of determining whether $\iota(G, F) \leq t$ is NP-complete. In Section~\ref{sec:mindeg}, for any graph $F$, we establish a bounded range, in terms of the minimum degree $\delta(G)$ of $G$ and the orders of $F$ and $G$, for the largest possible value of $\iota(G,F)$ with $\delta(G)$ sufficiently large. For a set $\mathcal{F}$ of graphs and a positive integer $t$, let $t \mathcal{F}$ denote the set 
\[\{F_1 \cup \dots \cup F_t \colon F_1, \dots, F_t \mbox{ are pairwise vertex-disjoint $\mathcal{F}$-graphs}\}.\]
If $\mathcal{F} = \{F\}$, then we may represent $t\mathcal{F}$ by $tF$. In Section~\ref{tFsection}, we introduce the problem of bounding $\iota(G, t \mathcal{F})$, focusing on how $\iota(G,t\mathcal{F})$ relates to $\iota(G,\mathcal{F})$, and exhibiting in particular an interesting connection with the Erd\H os--P\' osa property \cite{EP}.

\section{Complexity of $F$-isolation} \label{sec:complexity}

Let $F$ be an arbitrary graph. Various known results on the complexity of domination problems can be generalized to complexity results for $F$-isolation problems in a universal way. The \emph{dominating set problem} (DSP) is to determine whether $\gamma(G) \leq t$ for a given graph $G$ and an integer $t$. We will refer to the problem of determining whether $\iota(G, F) \leq t$ as the \emph{$F$-isolating set problem} ($F$-ISP). It is well known that DSP is NP-complete \cite{GJ}. Suppose that $F$ is connected. In this section, we show that $F$-ISP is NP-complete. It is also known that DSP is NP-complete for planar graphs \cite{GJ}, for chordal graphs \cite{BJ} and for bipartite graphs \cite{CN}. This fact was recently used in \cite{CLWX} for establishing that $P_k$-ISP is also NP-complete for planar graphs and for chordal graphs. We generalize this by constructing a graph $C(G,F)$ (for any graph $G$) and showing that for any set $\mathcal{G}$ of graphs for which DSP is NP-complete, $F$-ISP is NP-complete for $\{C(G,F) \colon G \in \mathcal{G}\}$. 

\begin{construction}\label{const1} \emph{For any $n, k \in \mathbb{N}$ and any graphs $G$ and $F$ with $|V(G)| = n$ and $|V(F)| = k$, we construct %an \emph{$F$-corona of $G$}, denoted by 
a graph $C(G,F)$ as follows. Let $v_1, \dots, v_n$ be the distinct vertices of $G$, let $u \in V(F)$, and let $F_1, \dots, F_n$ be pairwise vertex-disjoint copies of $F$ such that for each $i \in [n]$, $V(F_i) \cap V(G) = \{v_i\}$ and $v_i$ is the vertex of $F_i$ corresponding to $u$. Let $C(G,F)$ be the graph with vertex set $\bigcup_{i = 1}^n V(F_i)$ and edge set $E(G) \cup \bigcup_{i = 1}^n E(F_i)$. }
\end{construction}

\begin{lemma} \label{FISPlemma1} For any connected graph $F$, 
\[\iota(C(G,F), F) = \gamma(G).\]
\end{lemma}
\textbf{Proof.} Let $D$ be a smallest dominating set of $G$. Since $V(G) = N_G[D] \subseteq N_{C(G,F)}[D]$, for each component $H$ of $C(G,F) - N_{C(G,F)}[D]$, we have $V(H) \subseteq V(F_i) \setminus \{v_i\}$ for some $i \in [n]$. Thus, since $F$ is connected, $D$ is an $F$-isolating set of $G$. Therefore, we have $\iota(C(G,F), F) \leq |D| = \gamma(G)$.

Now let $D$ be a smallest $F$-isolating set of $C(G,F)$. Then, $N[D] \cap V(F_i) \neq \emptyset$ for each $i \in [n]$. For each $i \in [n]$, let $D_i = \{v_i\}$ if $D \cap V(F_i) \neq \emptyset$, and let $D_i = \emptyset$ if $D \cap V(F_i) = \emptyset$. Let $D' = \bigcup_{i=1}^n D_i$. Consider any $j \in [n]$ such that $v_j \notin D'$. Then, $D \cap V(F_j) = \emptyset$. Since $N[D] \cap V(F_j) \neq \emptyset$, $v_j \in N_G[D \cap V(G)]$. Since $D \cap V(G) \subseteq D'$, $v_j \in N_G[D']$. Thus, $D'$ is a dominating set of $G$. Since $\iota(C(G,F), F) = |D| \geq |D'| \geq \gamma(G) \geq \iota(C(G,F), F)$, the result follows.~\hfill{$\Box$}  

\begin{lemma} \label{FISPlemma2} For any graph $F$, $F$-ISP is in NP.
\end{lemma}
\textbf{Proof.} Let $k = |V(F)|$ and $m = |E(F)|$. Given an $n$-vertex graph $G$ and an integer $t \geq 0$, let $D \in {V(G) \choose t}$. We delete $N[D]$ from $G$. Let $G'$ be the resulting graph. Then, $D$ is an $F$-isolating set of $G$ if and only if $G'[X]$ contains no copy of $F$ for each $X \in {V(G') \choose k}$. This can checked in at most $c(F)n^k$ steps for some integer $c(F)$ depending only on $F$.~\hfill{$\Box$}

\begin{theorem} \label{FISP1} For any connected graph $F$, $F$-ISP is NP-complete.
\end{theorem}
\textbf{Proof.} Consider an instance of DSP. We are given a graph $G$ and an integer $t \geq 0$, and we need to check if $\gamma(G) \leq t$. We construct $C(G,F)$, which can be clearly done in polynomial time. By Lemma~\ref{FISPlemma1}, $\gamma(G) \leq t$ if and only if $\iota(C(G,F),F) \leq t$. Together with Lemma~\ref{FISPlemma2}, this yields the result.~\hfill{$\Box$}
\\

For any set $\mathcal{G}$ of graphs, let $C(\mathcal{G},F)$ denote the set $\{C(G,F) \colon G \in \mathcal{G}\}$. By the argument for Theorem~\ref{FISP1}, we obtain the above-mentioned stronger result.

\begin{theorem} \label{FISP2} If $F$ is a connected graph, $\mathcal{G}$ is a set of graphs, and DSP is NP-complete for $\mathcal{G}$, then $F$-ISP is NP-complete for $C(\mathcal{G},F)$.
\end{theorem}

\begin{cor}[\cite{CLWX}] For any integer $k \geq 1$, \\ 
(a) $P_k$-ISP is NP-complete for planar graphs, \\
(b) $P_k$-ISP is NP-complete for for chordal graphs.
\end{cor}
\textbf{Proof.} Let $\mathcal{G}_1$ be the set of planar graphs. Let $\mathcal{G}_2$ be the set of chordal graphs. As mentioned above, DSP is NP-complete for both $\mathcal{G}_1$ and $\mathcal{G}_2$. By Theorem~\ref{FISP2}, $P_k$-ISP is NP-complete for both $C(\mathcal{G}_1, P_k)$ and $C(\mathcal{G}_2, P_k)$. Let $G$ be a graph. Clearly, $C(G,P_k)$ is planar if $G$ is planar, and $C(G,P_k)$ is chordal if $G$ is chordal. Thus, $C(\mathcal{G}_1, P_k) \subseteq \mathcal{G}_1$ and $C(\mathcal{G}_2, P_k) \subseteq \mathcal{G}_2$. The result follows.~\hfill{$\Box$}

\begin{cor} For any connected bipartite graph $F$, $F$-ISP is NP-complete for bipartite graphs.
\end{cor}
\textbf{Proof.} Let $\mathcal{G}$ be the set of bipartite graphs. As mentioned above, DSP is NP-complete for $\mathcal{G}$. By Theorem~\ref{FISP2}, $F$-ISP is NP-complete for $C(\mathcal{G}, F)$. Clearly, if $G$ is a bipartite graph, then $C(G, F)$ is bipartite. Thus, $C(\mathcal{G}, F) \subseteq \mathcal{G}$. The result follows.~\hfill{$\Box$}

\section{Bounds on the $F$-isolation number in terms of the minimum degree} \label{sec:mindeg}

For any integer $d \geq 1$ and any graph $F$, let
\[\iota(d, F) = \inf\{\alpha \colon \iota(G, F) \leq \alpha |V(G)| \mbox{ for every graph $G$ with } \delta(G) \geq d\}.\]
Let $\alpha_d = \frac{1 + \ln(d+1)}{d+1}$. The following is a classical result of Arnautov \cite{Arnautov}, Lov\'{a}sz \cite{Lovasz} and Payan \cite{Payan} (see also \cite[Theorem~1.2.2]{AS}).  
\begin{theorem} \label{dombound1}
If $G$ is an $n$-vertex graph with $\delta(G) \geq d \geq 1$, then $\gamma(G) \leq \alpha_d n$.
\end{theorem}
Thus, since we trivially have $\iota(G,F) \leq \gamma(G)$ for any graph $G$, 
\begin{equation} \iota(d,F) \leq \alpha_d . \label{iotalpha}
\end{equation}
On the other hand, Caro and Hansberg \cite[Theorem~3.5 (iii)]{CaHa17} proved that if $F = K_{1,k}$, then
\begin{equation} \iota(d, F) \geq \frac{(1 + o(1))}{|V(F)|} \alpha_d , \label{CHineq}
\end{equation} 
where $o(1)$ is the standard notation representing any function of $d$ that tends to $0$ as $d \rightarrow \infty$. In this section, we show that (\ref{CHineq}) holds for any $F$.

In \cite{AW}, Alon and Wormald established the following probabilistic result as part of the proof of Theorem 1.2 in that paper.

\begin{lemma}[\cite{AW}] \label{AWlemma} For any real number $c < 1$, the expected number of dominating sets of size $cn(1 + o(1))(\ln d)/d$ of a random $d$-regular $n$-vertex graph tends to $0$ as $d \rightarrow \infty$.
\end{lemma}

For any $n, d \in \mathbb{N}$, we say that $(n,d)$ is \emph{good} if $n > d$ and $nd$ is even. It is well known that a $d$-regular $n$-vertex graph exists if and only if $(n,d)$ is good. For any good $(n,d)$, let
\[\gamma(n,d) = \max\{\gamma(G) \colon G \mbox{ is a $d$-regular $n$-vertex graph}\}.\]

We restate Lemma~\ref{AWlemma} in a form that is suitable for our purpose.

\begin{lemma}\label{maxgamma} $\gamma(n, d) = (1 + o(1))\frac{\ln d}{d}n$.
\end{lemma}
\textbf{Proof.} Let $\varepsilon$ be a real number such that $0 < \varepsilon < 1$. Since $(1 + \ln(d+1))/(d+1) \sim (\ln d) / d$, there exists some $d_0 \in \mathbb{N}$ such that 
\begin{equation} \frac{1 + \ln(d+1)}{d+1} \leq (1 + \varepsilon)\frac{\ln d}{d} \label{AWlemma.1}
\end{equation}
for any $d \geq d_0$. Let $c = 1 - \varepsilon / 2$. Let $b$ be a function satisfying $b(d) = o(1)$. Then, for some $d_1 \in \mathbb{N}$, $|b(d)| < \varepsilon /2$ for any $d \geq d_1$.  For any good $(n,d)$, let $f(n,d) = cn(1 + b(d))(\ln d)/d$, and let $g(n,d)$ be the expected number of dominating sets of size $f(n,d)$ of a random $d$-regular $n$-vertex graph. By Lemma~\ref{AWlemma}, for some $d_2 \in \mathbb{N}$, $g(d) < 1$ for any $d \geq d_2$. Let $d \geq \max\{d_0, d_1, d_2\}$. Let $n \in \mathbb{N}$ such that $(n,d)$ is good. By the probabilistic pigeonhole principle, there exists a $d$-regular $n$-vertex graph $G$ with at most $g(n,d)$ dominating sets of size $f(n,d)$. Adding vertices to a dominating set yields another dominating set, so if a graph has no dominating set of a certain size, then it has no dominating set of a smaller size. Since $g(n,d) < 1$, $G$ has no dominating set of size at most $f(n,d)$, so $\gamma(G) > f(n,d)$. We have
\begin{align} \frac{d}{n\ln d}\gamma(G) &> \frac{d}{n\ln d}f(n,d) = (1 - \frac{\varepsilon}{2})(1 + b(d)) = 1 - \frac{\varepsilon}{2} + (1 - \frac{\varepsilon}{2})b(d) \nonumber \\
&> 1 - \frac{\varepsilon}{2} - (1 - \frac{\varepsilon}{2})\frac{\varepsilon}{2} > 1 - \varepsilon, \nonumber
\end{align}
so $\gamma(n, d) > (1 - \varepsilon)\frac{\ln d}{d}n$. By (\ref{AWlemma.1}) and Theorem~\ref{dombound1}, $\gamma(n, d) < (1 + \varepsilon)\frac{\ln d}{d}n$. The result follows.~\hfill{$\Box$}

\begin{defn} \emph{For any graphs $G$ and $F$, the \emph{Cartesian product of $G$ and $F$}, denoted by $G \Box F$, is the graph with vertex set $V(G) \times V(F)$ and edge set} \[{\bigg (} \bigcup_{v \in V(G)} \{(v,w)(v,w') \colon ww' \in E(F)\} {\bigg )} \cup {\bigg (} \bigcup_{w \in V(F)} \{(v,w)(v',w) \colon vv' \in E(G)\} {\bigg )}.\]
\end{defn}

The following is a straightforward well-known fact.

\begin{prop} \label{mindegGF} $\delta(G \Box F) = \delta(G) + \delta(F)$.
\end{prop}

\begin{lemma}\label{Carprod} $\iota(G \Box F, F) \geq \gamma(G)$.
\end{lemma}
\textbf{Proof.} Let $H = G \Box F$. Let $D$ be a smallest $F$-isolating set of $H$. Let $D^* = \{v \in V(G) \colon (v,w) \in D \mbox{ for some } w \in V(F)\}$. For any $w \in V(F)$, let $G_w$ be the $G$-copy given by $H [V(G) \times \{w\}\}$. For any $v \in V(G)$, let $F_v$ be the $F$-copy given by $H [\{v\} \times V(F)\}$. Thus, $N_H[(x,y)] \cap V(F_v) \neq \emptyset$ for some $(x,y) \in D$, and hence $x \in D^*$. Suppose $v \neq x$. Since $N_H[(x,y)] = N_{F_x}[(x,y)] \cup N_{G_y}[(x,y)]$, we obtain $(v,y) \in N_{G_y}[(x,y)]$, so $v \in N[x] \subseteq N[D^*]$. We have $\iota(H, F) = |D| \geq |D^*| \geq \gamma(G)$.~\hfill{$\Box$}

\begin{theorem}\label{mindegthm} For any integer $d \geq 1$ and any graph $F$,
\[\frac{(1 + o(1))}{|V(F)|} \alpha_d \leq \iota(d, F) \leq \alpha_d . \]
\end{theorem}
\textbf{Proof.} Let $k = |V(F)|$. Let $G_d$ be a $d$-regular $n_d$-vertex graph with $\gamma(G_d) = \gamma(n_d, d)$. Let $H_d = G_d \Box F$. 
%Let $b_d = \frac{1 + \ln(d + d' + 1)}{d + d' + 1}$. 
By Proposition~\ref{mindegGF}, $\delta(H_d) = d + \gamma(F)$. By Lemmas~\ref{Carprod} and \ref{maxgamma}, $\iota(H_d, F) \geq \gamma(G_d) = (1 + o(1))\frac{\ln d}{d}n_d$. Thus,
\begin{align*} \frac{\iota(H_d,F)}{\alpha_{d} |V(H_d)|} &\geq \frac{(1 + o(1))\frac{\ln d}{d}n_d}{\frac{1 + \ln(d+1)}{d+1}kn_d} = \frac{1}{k}(1 + o(1))\frac{d+1}{d}\frac{\ln d}{1 + \ln(d+1)} \rightarrow \frac{1}{k}
\end{align*}
as $d \rightarrow \infty$. Therefore, $\iota(H_d,F) \geq \frac{1}{k}(1 + o(1))\alpha_d |V(H_d)|$, and hence $\iota(d, F) \geq \frac{(1 + o(1))}{k} \alpha_d$. Together with (\ref{iotalpha}), this yields the result.~\hfill{$\Box$}
\\

For any integer $d \geq 1$ and any graph $F$, let
\[\iota_{\rm reg}(d, F) = \inf\{\alpha \colon \iota(G, F) \leq \alpha |V(G)| \mbox{ for every regular graph $G$ with } \delta(G) \geq d\}.\]
Clearly, if $F$ is regular, then $G \Box F$ is regular. Thus, the proof of Theorem~\ref{mindegthm} yields the following stronger result for the case where $F$ is regular.

\begin{theorem}\label{mindegthmreg} For any integer $d \geq 1$ and any regular graph $F$,
\[\frac{(1 + o(1))}{|V(F)|} \alpha_d \leq \iota_{\rm reg}(d, F) \leq \iota(d, F) \leq \alpha_d . \]
\end{theorem}

\section{Isolation of $t\mathcal{F}$}\label{tFsection}

Following the celebrated Erd\H os--P\' osa (EP) Theorem \cite{EP}, $\mathcal{F}$ is said to have the \emph{EP property} if there exists a function $f_{\mathcal{F}} \colon \mathbb{N} \rightarrow \mathbb{N}$ such that for every $t \in \mathbb{N}$ and every graph $G$ containing no $t \mathcal{F}$-graph, $V(G)$ has a subset $S$ such that $|S| \leq f_{\mathcal{F}}(t)$ and $G - S$ contains no $\mathcal{F}$-graph (equivalently, $S$ intersects the vertex sets of the $\mathcal{F}$-graphs contained by $G$). The EP Theorem \cite{EP} tells us that the set $\mathcal{C}$ of cycles has the EP property and that $f_{\mathcal{C}}(t) = \Theta(t \log t)$. For more on the EP property, see, for example, \cite{AGHK, CHJR, FGW, Lovasz2, RS}.

Almost all of the isolation problems treated to date concern bounding $\iota(G, \mathcal{F})$ from above for some set $\mathcal{F}$ of connected graphs. If $t \geq 2$, then each member of $t\mathcal{F}$ is a disconnected graph (having exactly $t$ components, which are members of $\mathcal{F}$). Let $\gamma(\mathcal{F})$ denote $\max\{\gamma(F) \colon F \in \mathcal{F}\}$. Trivially, $\iota(G, t\mathcal{F}) \leq \iota(G, \mathcal{F}) \leq \gamma(G)$. We show that, on the other hand, $\iota(G, t\mathcal{F})$ can be bounded from below using $\iota(G, \mathcal{F})$ and $\gamma(\mathcal{F})$, and if $\mathcal{F}$ has the EP property, then we can use $f_{\mathcal{F}}(t)$ instead of $\gamma(\mathcal{F})$.

\begin{theorem}\label{tF} For any graph $G$ and any set $\mathcal{F}$ of graphs,
\[\iota(G,\mathcal{F}) - (t-1)\gamma(\mathcal{F}) \leq \iota(G,t\mathcal{F}) \leq \iota(G,\mathcal{F}).\]
Moreover, for any $k \geq 1$ and $q \geq 1$, there exist two graphs $G$ and $H$ such that $\iota(G, tK_k) = \iota(G, K_k) = \gamma(G) = q = \iota(H, tK_k) = \iota(H, K_k) - (t-1)\gamma(K_k)$ and $\iota(H, K_k) = \gamma(H)$. 
\end{theorem} 
\textbf{Proof.} Recall that $\iota(G, t\mathcal{F}) \leq \iota(G, \mathcal{F})$. Let $D$ be a $t\mathcal{F}$-isolating set of $G$ of size $\iota(G, t\mathcal{F})$. Let $H = G - N_G[D]$. Let $\mathcal{H}$ be a largest set of pairwise vertex-disjoint $\mathcal{F}$-graphs contained by $H$. Let $F_1, \dots, F_h$ be the distinct members of $\mathcal{H}$. By the choice of $D$, $h \leq t-1$. For each $i \in [h]$, let $D_i$ be a dominating set of $F_i$ of size $\gamma(F_i)$. Let $D' = \bigcup_{i=1}^h D_i$. We have $\bigcup_{i=1}^h V(F_i) = \bigcup_{i=1}^h N_{F_i}[D_i] \subseteq \bigcup_{i=1}^h N_H[D_i] = N_H[D']$. Suppose that $H - N_H[D']$ contains an $\mathcal{F}$-graph $F$. Then, we have that $\mathcal{H} \cup \{F\}$ is a set of pairwise vertex-disjoint $\mathcal{F}$-graphs contained by $H$, which contradicts the choice of $\mathcal{H}$. Therefore, $H - N_H[D']$ contains no copy of $F$. Since $V(H - N_H[D']) = V(G) \setminus (N_G[D] \cup N_H[D']) \supseteq V(G) \setminus (N_G[D] \cup N_G[D']) = V(G - N_G[D \cup D'])$, $D \cup D'$ is an $\mathcal{F}$-isolating set of $G$. We have $\iota(G,\mathcal{F}) \leq |D| + |D'| = \iota(G, t\mathcal{F}) + \sum_{i=1}^h \gamma(F_i) \leq \iota(G, t\mathcal{F}) + h \gamma(\mathcal{F})$.

We now recall \cite[Construction~1]{BFK2} to prove the second part. Let $k \geq 1$ and $q \geq 1$. Suppose $\mathcal{F} = \{K_k\}$. Let $v_1, \dots, v_q$ be the distinct elements of a path $Q$, let $G_1, \dots, G_q$ be copies of $K_{tk}$ such that $G_1, \dots, G_q$ and $Q$ are pairwise vertex-disjoint, and for each $i \in [q]$, let $w_i \in V(G_i)$ and let $G_i'$ be the graph with $V(G_i') = \{v_i\} \cup V(G_i)$ and $E(G_i') = \{v_iw_i\} \cup E(G_i)$. Suppose that $G$ is the graph with $V(G) = \bigcup_{i=1}^q V(G_i')$ and $E(G) = E(Q) \cup \bigcup_{i=1}^q E(G_i')$. Let $D$ be a smallest $t\mathcal{F}$-isolating set of $G$. For each $i \in [q]$, $G_i$ contains a $t\mathcal{F}$-graph, and $N_G[V(G_i)] = V(G_i')$, so $D \cap V(G_i') \neq \emptyset$. Since $V(G_1'), \dots, V(G_q')$ partition $V(G)$, $|D| = \sum_{i=1}^q |D \cap V(G_i')| \geq q$. We have $q \leq |D| = \iota(G, t\mathcal{F}) \leq \iota(G, \mathcal{F}) \leq \gamma(G)$. Since $\{w_i \colon i \in [q]\}$ is a dominating set of $G$, $\iota(G, t\mathcal{F}) = \iota(G, \mathcal{F}) = \gamma(G) = q$. Now let $H_1, \dots, H_t$ be cliques such that $|V(H_1)| = k$, $|V(H_i)| = k+1$ for each $i \in [t] \setminus \{1\}$, and $H_1, \dots, H_t$ and $G$ are pairwise vertex-disjoint. For each $i \in [t]$, let $x_i \in V(H_i)$ and let $H_i'$ be the graph with $V(H_i') = \{v_q\} \cup V(H_i)$ and $E(H_i') = \{v_qx_i\} \cup E(H_i)$. Let $H$ be the graph with $V(H) = (V(G) \setminus V(G_q)) \cup \bigcup_{i=1}^t V(H_i')$ and $E(H) = (E(G) \setminus E(G_q')) \cup \bigcup_{i=1}^t E(H_i')$. We have 
\begin{equation} \mbox{$H_1 \simeq K_k$ and $N_H[V(H_1)] = \{v_q\} \cup V(H_1)$.} \label{tF1a}
\end{equation} 
For each $i \in [t] \setminus \{1\}$, 
\begin{equation} \mbox{$H_i - x_i \simeq K_k$ and $N_H[V(H_i - x_i)] = V(H_i)$.} \label{tF1b}
\end{equation} 
Let $W = V(H_1) \cup \bigcup_{i = 2}^t V(H_i - x_i)$ and $W' = \{v_q\} \cup \bigcup_{i = 1}^t V(H_i)$. Then, 
\begin{equation} \mbox{$H[W] \in t\mathcal{F}$ and $N_H[W] = W'$.} \label{tF2}
\end{equation}
Let $X$ be a smallest $\mathcal{F}$-isolating set of $H$, and let $Y$ be a smallest $t\mathcal{F}$-isolating set of $H$. By the argument for $D$, for each $i \in [q-1]$, $X \cap V(G_i') \neq \emptyset$ and $Y \cap V(G_i') \neq \emptyset$. By (\ref{tF1a}) and (\ref{tF1b}), $X \cap (\{v_q\} \cup V(H_1)) \neq \emptyset$ and $X \cap V(H_i) \neq \emptyset$ for each $i \in [t] \setminus \{1\}$. By (\ref{tF2}), $Y \cap W' \neq \emptyset$. Therefore, we have $q + t - 1 \leq |X| = \iota(H, \mathcal{F}) \leq \gamma(H)$ and $q \leq |Y| = \iota(H, t\mathcal{F})$. Since $\{w_i \colon i \in [q-1]\} \cup \{x_j \colon j \in [t]\}$ is a dominating set of $H$, $\iota(H, \mathcal{F}) = \gamma(H) = q+t-1$. Since $\{w_i \colon i \in [q-1]\} \cup \{v_q\}$ is a $t\mathcal{F}$-isolating set of $H$, $\iota(H, t\mathcal{F}) = q$. Since $\gamma(\mathcal{F}) = 1$, the result follows.~\hfill{$\Box$}

\begin{theorem}\label{tFEP} If $G$ is a graph, $\mathcal{F}$ is a set of graphs, and $\mathcal{F}$ has the EP property, then
\[\iota(G,\mathcal{F}) - f_\mathcal{F}(t) \leq \iota(G, t\mathcal{F}) \leq \iota(G, \mathcal{F}).\]
\end{theorem} 
\textbf{Proof.} The second inequality is given by Theorem~\ref{tF}. Let $D$ be a $t\mathcal{F}$-isolating set of $G$ of size $\iota(G, t\mathcal{F})$. Let $H = G - N[D]$. Then, $H$ contains no $t \mathcal{F}$-graph. Since $\mathcal{F}$ has the EP property, $V(H)$ has a subset $S$ such that $|S| \leq f_{\mathcal{F}}(t)$ and $H - S$ contains no $\mathcal{F}$-graph. Thus, $H-N[S]$ has no $\mathcal{F}$-graph. Consequently, $D \cup S$ is an $\mathcal{F}$-isolating set of $G$, and hence $\iota(G, \mathcal{F}) \leq |D| + |S| \leq \iota(G, t\mathcal{F}) + f_{\mathcal{F}}(t)$.~\hfill{$\Box$}
\\

Note that the proof of Theorem~\ref{tFEP} immediately raises the question of whether the bound can be improved by using $N[S]$ directly when treating $H$. We now show that this is not the case if $\mathcal{F}$ is the set $\mathcal{C}$ of cycles. 

If $S$ is a subset of $V(G)$ that intersects the vertex sets of the $\mathcal{F}$-graphs contained by $G$, then we call $S$ an \emph{$\mathcal{F}$-hitting set of $G$}. Note that $S$ is an $\mathcal{F}$-hitting set of $G$ if and only if $G-S$ contains no $\mathcal{F}$-graph. In the literature, a $\mathcal{C}$-hitting set of $G$ is called a \emph{decycling set of $G$}, and the size of a smallest decycling set of $G$ is called the \emph{decycling number of $G$} and is denoted by $\nabla(G)$. Let ${\rm S}_2(G)$ denote the graph obtained by subdividing each edge of $G$ twice, meaning that each edge $xy$ of $G$ is replaced by $xx'$, $x'y'$ and $y'y$ for some two new vertices $x'$ and $y'$. Let $\mathcal{C}(G)$ denote the set of cycles contained by $G$. Clearly, 
\begin{equation} \mathcal{C}({\rm S}_2(G)) = \{{\rm S}_2(F) \colon F \in \mathcal{C}(G)\}. \label{CS2}
\end{equation}

\begin{theorem}\label{decisol} For any graph $G$, 
\[\nabla(G) = \nabla({\rm S}_2(G)) = \iota({\rm S}_2(G), \mathcal{C}).\]
\end{theorem}
\textbf{Proof.} Let $H = {\rm S}_2(G)$. Let $D_G$ be a smallest decycling set of $G$. Let $D_H$ be a smallest decycling set of $H$ such that $|D_H \setminus V(G)|$ is minimum. 

By (\ref{CS2}), $D_G$ is a decycling set of $H$, so $\nabla(H) \leq \nabla(G)$. Suppose $D_H \setminus V(G) \neq \emptyset$. Then, $x' \in D_H$ for some $xy \in E(G)$. We have $x \in V(F)$ for each $F \in \mathcal{C}(H)$ with $x' \in V(F)$. This gives us that $(D_H \setminus \{x'\}) \cup \{x\}$ is a decycling set of $H$, contradicting the choice of $D_H$. Therefore, $D_H \setminus V(G) = \emptyset$, and hence $D_H \subseteq V(G)$. By (\ref{CS2}), $D_H$ is a decycling set of $G$, so $\nabla(G) \leq \nabla(H)$. Together with $\nabla(H) \leq \nabla(G)$, this yields $\nabla(G) = \nabla(H)$.

Trivially, every decycling set of $H$ is a $\mathcal{C}$-isolating set of $H$, so $\iota(H, \mathcal{C}) \leq \nabla(H)$. Let $D$ be a smallest $\mathcal{C}$-isolating set of $H$ such that $|D \setminus V(G)|$ is minimum. Suppose $D \setminus V(G) \neq \emptyset$. Then, $x' \in D$ for some $xy \in E(G)$. We have $N_H[x'] = \{x, x', y'\}$. For each $F \in \mathcal{C}(H)$ with $N_H[x'] \cap V(F) \neq \emptyset$, we actually have $N_H[x'] \subseteq V(F)$.  This gives us that $(D \setminus \{x'\}) \cup \{x\}$ is a $\mathcal{C}$-isolating set of $H$, contradicting the choice of $D$. Therefore, $D \setminus V(G) = \emptyset$, and hence $D \subseteq V(G)$. Consider any $F \in \mathcal{C}(G)$. By (\ref{CS2}), ${\rm S}_2(F) \in \mathcal{C}(H)$. Suppose $D \cap V(F) = \emptyset$.  Then, for each $v \in D$, we have $v \notin V(F)$, and hence $N_H[v] \cap {\rm S}_2(F) = \emptyset$. Thus, we have $N_H[D] \cap {\rm S}_2(F) = \emptyset$, a contradiction as $D$ is a $\mathcal{C}$-isolating set of $H$. Thus, $D$ is a decycling set of $G$, and hence $\nabla(G) \leq \iota(H, \mathcal{C})$. Together with $\nabla(G) = \nabla(H)$ and $\iota(H, \mathcal{C}) \leq \nabla(H)$, this yields $\nabla(H) = \iota(H, \mathcal{C})$.~\hfill{$\Box$}
\\

Theorems~\ref{tF} and \ref{decisol} yield the following complexity results.

\begin{rem} \emph{The \emph{minimum dominating set problem} (MDSP) is to determine $\gamma(G)$ for a given graph $G$. MDSP is NP-hard \cite{GJ}. For any graph $H$, we will refer the to problem of determining $\iota(G,H)$ as the \emph{minimum $H$-isolating set problem} (Min $H$-ISP). Let $F$ be a connected graph. Let $t$ be a positive integer. Let $a = (t-1)\gamma(F)$. By Lemma~\ref{FISPlemma1} and Theorem~\ref{tF1b}, we have $\gamma(G) = \iota(C(G,F), F) \leq \iota(C(G,F), tF) + a$ and $\iota(C(G,F), tF) \leq \iota(C(G,F), F) = \gamma(G)$, so $\gamma(G) \leq \iota(C(G,F), tF) + a \leq \gamma(G) + a$.  Thus, $\iota(C(G,F), tF) + a$ approximates $\gamma(G)$. It follows by \cite[Corollary~4]{DS} (which is stated for the approximation version of the minimum set cover problem, but it is well known that this implies a similar result for the approximation version of MDSP) that Min $tF$-ISP is NP-hard.}
\end{rem}

\begin{rem} \emph{The \emph{feedback vertex set problem} (FVSP) is to determine whether $\nabla(G) \leq t$ for a given graph $G$ and an integer $t$. We will refer to the problem of determining whether $\iota(G, \mathcal{C}) \leq t$ as the \emph{$\mathcal{C}$-isolating set problem} ($\mathcal{C}$-ISP). FVSP is NP-complete \cite{Karp}. We now show that together with a straightforward generalization of Theorem~\ref{decisol}, this gives us that for any integer $r \geq 1$, $\mathcal{C}$-ISP is NP-hard for the set $\mathcal{B}_r$ of bipartite graphs whose cycles are of length $0 \bmod 2r$ (where mod is the usual modulo operation). Similarly to ${\rm S}_2(G)$, for any integer $h \geq 0$, let ${\rm S}_{h}(G)$ denote the graph obtained by subdividing each edge of $G$ $h$ times. Then, each cycle of ${\rm S}_{2r+1}(G)$ is of length $0 \bmod 2r$. A well-known result in the literature is that a graph is bipartite if and only if its cycles are of even length. Thus, ${\rm S}_{2r+1}(G) \in \mathcal{B}_r$. It is easy to see that Theorem~\ref{decisol} generalizes as follows:  for $h \geq 2$, $\nabla(G) = \nabla({\rm S}_{h}(G)) = \iota({\rm S}_{h}(G), \mathcal{C})$. Thus, $\nabla(G) \leq t$ if and only if $\iota({\rm S}_{2r+1}(G), \mathcal{C}) \leq t$.}
\end{rem}


\footnotesize
\begin{thebibliography}{}

\bibitem{AGHK} J. Ahn, J.P. Gollin, T. Huynh, O. Kwon, A coarse Erd\H os--P\' osa theorem, arXiv:2407.05883 [math.CO].

\bibitem{AS} N. Alon, J. Spencer, The Probabilistic Method, Wiley-Interscience Series in Discrete Mathematics and Optimization, third edition, John Wiley \& Sons, Inc., 2008.

\bibitem{AW} N. Alon and N. Wormald, High degree graphs contain large-star factors, in Fete of Combinatorics and Computer Science, Bolyai Society Mathematical Studies, Vol. 20, Springer, Berlin--Heidelberg, 2010, pp. 9--21.

\bibitem{Arnautov} V. I. Arnautov, Estimations of the external stability number of a graph by means of the minimal degree of vertices, Prikl. Mat. Programm. V 11 (1974), 3--8 (in Russian).

\bibitem{BJ} K.S. Booth, J.H. Johnson, Dominating sets in chordal graphs, SIAM J. Comput. 11 (1982) 191--199.

\bibitem{Borgcycles} P. Borg, Isolation of cycles, Graphs Combin. 36 (2020), 631--637.

\bibitem{Borgcon} P. Borg, Isolation of connected graphs, Discrete Appl. Math. 339 (2023), 154--165.

\bibitem{Borgisdom} P. Borg, Proof of a conjecture on isolation of graphs dominated by a vertex, Discrete Appl. Math. 371 (2025), 247-253.

\bibitem{Borgrc2} P. Borg, Isolation of regular graphs and $k$-chromatic graphs, Mediterr. J. Math. 21 (2024), paper 148.

\bibitem{Borgrsc} P. Borg, Isolation of regular graphs, stars and $k$-chromatic graphs, Discrete Math. 349 (2026), paper 114706.

\bibitem{BFK} P. Borg, K. Fenech, P. Kaemawichanurat, Isolation of $k$-cliques, Discrete Math. 343 (2020), paper 111879.

\bibitem{BFK2} P. Borg, K. Fenech, P. Kaemawichanurat, Isolation of $k$-cliques II, Discrete Math. 345 (2022), paper 112641.

\bibitem{BCHH} R. Boutrig, M. Chellali, T.W. Haynes, S.T. Hedetniemi, Vertex-edge domination in graphs, Aequ. Math. 90 (2016), 355--366.

\bibitem{BG} G. Boyer, W. Goddard, Disjoint isolating sets and graphs with maximum isolation number, Discrete Appl. Math. 356 (2024), 110--116.

\bibitem{CHJR} W. Cames van Batenburg, T. Huynh, G. Joret, J. Raymond, A tight Erd\H{o}s--P\'{o}sa function for planar minors,
Adv. Comb. (2019), paper 2.

\bibitem{CaHa15} Y. Caro and A. Hansberg, Isolation in graphs, Electronic Notes in Discrete Mathematics 50 (2015), 465--470.

\bibitem{CaHa17} Y. Caro and A. Hansberg, Partial domination - the isolation number of a graph, Filomat 31 (2017), 3925--3944.

\bibitem{CN} G.J. Chang and G.L. Nemhauser, The $k$-domination and $k$-stability problems in sun-free chordal graphs, SIAM J. Algebraic Discrete Methods 5 (1984), 332--345.

\bibitem{CCZ} S. Chen, Q. Cui, J. Zhang, A characterization of graphs with maximum cycle isolation number, Discrete Appl. Math. 366 (2025), 161--175.

\bibitem{CCZ2} S. Chen, Q. Cui, L. Zhong, A characterization of graphs with maximum $k$-clique isolation number, Discrete Math. 348 (2025), 114531.

\bibitem{CLWX} J. Chen, Y. Liang, C. Wang and S. Xu, Algorithmic aspects of $\{P_k\}$-isolation in graphs and extremal graphs for a $\{P_3\}$-isolation bound, Inf. Process. Lett. 187 (2025), paper 106521.

\bibitem{C} E.J. Cockayne, Domination of undirected graphs -- A survey, in: Lecture Notes in Mathematics, Volume 642, Springer, 1978, 141--147.

\bibitem{CH} E.J. Cockayne, S.T. Hedetniemi, Towards a theory of domination in graphs, Networks 7 (1977), 247--261.

\bibitem{CZ} Q. Cui, J. Zhang, A sharp upper bound on the cycle isolation number of graphs, Graphs Combin. 39 (2023), paper 117.

\bibitem{DS} I. Dinur, D. Steurer, Analytical approach to parallel repetition, Proceedings of the 2014 ACM Symposium on Theory of Computing, ACM Press, New York, 2014, 624--633.

\bibitem{EP} P. Erd\H os, L. P\' osa, On independent circuits contained in a graph, Canad. J. Math. 17 (1965), 347--352.

\bibitem{FJKR} J.F. Fink, M.S. Jacobson, L.F. Kinch, J. Roberts, On graphs having domination number half their order, Period. Math. Hungar. 16 (1985), 287--293.

\bibitem{FGW} S. Fiorini, G. Joret, D.R. Wood, Excluded forest minors and the Erd\H os--P\' osa property, Combin. Probab. Comput. 22 (2013), 700--721.

\bibitem{GJ} M.R. Garey, D.S. Johnson, Computers and intractability. A guide to the theory of NP-completeness, Ser. Books Math. Sci.,
W.H. Freeman and Co., San Francisco, CA, 1979.

\bibitem{HHS} T.W. Haynes, S.T. Hedetniemi, P.J. Slater, Fundamentals of Domination in Graphs, Marcel Dekker, Inc., New York, 1998.

\bibitem{HHS2} T.W. Haynes, S.T. Hedetniemi, P.J. Slater (Editors), Domination in Graphs: Advanced Topics, Marcel Dekker, Inc., New York, 1998.

\bibitem{HL} S.T. Hedetniemi, R.C. Laskar (Editors), Topics on Domination, in: Annals of Discrete Mathematics, Volume 48, North-Holland Publishing Co., Amsterdam, 1991, Reprint of Discrete Mathematics 86 (1990). %, no.~1--3.

\bibitem{HL2} S.T. Hedetniemi, R.C. Laskar, Bibliography on domination in graphs and some basic definitions of domination parameters, Discrete Math. 86 (1990), 257--277.

\bibitem{Karp} R.M. Karp, Reducibility among combinatorial problems, Complexity of Computer Computations, The IBM Research Symposia Series, Plenum Press, New York-London, 1972, pp.~85--103.

\bibitem{LMS} M. Lema\'{n}ska,  M. Mora, M.J. Souto-Salorio, Graphs with isolation number equal to one third of the order, Discrete Math. 347 (2024), paper 113903.

\bibitem{Lewis} J.R. Lewis, Vertex-edge and edge-vertex domination in graphs, Ph.D. Thesis, Clemson University, Clemson, 2007.

\bibitem{LHHF} J.R. Lewis, S.T. Hedetniemi, T.W. Haynes, G.H. Fricke, Vertex-edge domination, Util. Math. 81 (2010), 193--213.

\bibitem{Lovasz} L. Lov\'{a}sz, On decompositions of graphs, Studia Sci. Math Hungar. 1 (1966) 237--238.

\bibitem{Lovasz2} L. Lov\'{a}sz, On graphs not containing independent circuits, Mat. Lapok. 16 (1965), 289--299.

\bibitem{Ore} O. Ore, Theory of graphs, in: American Mathematical Society Colloquium Publications, Volume 38, American Mathematical Society, Providence, R.I., 1962.

\bibitem{Payan} C. Payan, Sur le nombre d'absorption d'un graphe simple, Cahiers Centre \'{E}tudes Recherche Op\'{e}r. 17 (1975), 307--317.

\bibitem{PX} C. Payan, N.H. Xuong, Domination-balanced graphs, J. Graph Theory 6 (1982), 23--32.

\bibitem{Peters} K.W. Peters, Theoretical and algorithmic results on domination and connectivity, Ph.D. Thesis, Clemson University, Clemson, 1986.

\bibitem{RS} N. Robertson, P. Seymour, Graph minors. V. Excluding a planar graph, J. Combin. Theory Ser. B 41 (1986), 92--114.

\bibitem{TJ} D.K. Thakkar, N.P. Jamvecha, About ve-domination in graphs, Ann. Pure Appl. Math. 14 (2017), 245--250.

\bibitem{West} D.B. West, Introduction to Graph Theory, second edition, Prentice Hall, 2001.

\bibitem{ZW2} G. Zhang, B. Wu, A note on the cycle isolation number of graphs, Bull. Malays. Math. Sci. Soc. 47 (2024), paper 57.

\bibitem{Z} P. \.{Z}yli\'{n}ski, Vertex-edge domination in graphs, Aequ. Math. 93 (2019), 735--742.

\end{thebibliography}
\end{document}